\documentclass[twoside,11pt]{article}

\usepackage{graphicx, subfigure}
\usepackage[top=1in,bottom=1in,left=1in,right=1in]{geometry}
\usepackage[sort&compress]{natbib} 
\usepackage{amsfonts, amsmath, amssymb, amsthm, constants, bbm}
\usepackage{MnSymbol}

\usepackage{color}
\definecolor{darkred}{RGB}{100,0,0}
\definecolor{darkgreen}{RGB}{0,100,0}
\definecolor{darkblue}{RGB}{0,0,150}

\usepackage{hyperref}
\hypersetup{colorlinks=true, linkcolor=darkred, citecolor=darkgreen, urlcolor=darkblue}
\usepackage{url}

\newtheorem{thm}{Theorem}
\newtheorem{prp}{Proposition}

\newtheorem{cor}{Corollary}

\newtheorem{rem}{Remark}

\def\beq{\begin{equation}} 
\def\eeq{\end{equation}}
\def\beqn{\begin{eqnarray*}}
\def\eeqn{\end{eqnarray*}}
\def\Bitem{\begin{itemize}\setlength{\itemsep}{.2in}}
\def\bitem{\begin{itemize}\setlength{\itemsep}{.05in}}
\def\eitem{\end{itemize}}
\def\Benum{\begin{enumerate}\setlength{\itemsep}{.2in}}
\def\benum{\begin{enumerate}\setlength{\itemsep}{.05in}}
\def\eenum{\end{enumerate}}
\def\bmult{\begin{multline*}}
\def\emult{\end{multline*}}
\def\bcenter{\begin{center}}
\def\ecenter{\end{center}}
\def\bframe{\begin{frame}}
\def\eframe{\end{frame}}

\newcommand{\thmref}[1]{Theorem~\ref{thm:#1}}
\newcommand{\prpref}[1]{Proposition~\ref{prp:#1}}
\newcommand{\corref}[1]{Corollary~\ref{cor:#1}}

\newcommand{\secref}[1]{Section~\ref{sec:#1}}

\DeclareMathOperator*{\argmin}{arg\, min}

\def\cG{\mathcal{G}}

\def\bZ{\mathbf{Z}}

\def\bz{\mathbf{z}}

\def\bbI{\mathbb{I}}

\def\bbR{\mathbb{R}}

\newcommand{\E}{\operatorname{\mathbb{E}}}
\renewcommand{\P}{\operatorname{\mathbb{P}}}
\newcommand{\Var}{\operatorname{Var}}

\def\eps{\varepsilon}

\def\1{\mathbbm{1}}
\newcommand{\IND}[1]{\bbI\{ #1 \}}

\definecolor{purple}{rgb}{0.4,.1,.9}

\begin{document}
\thispagestyle{empty}

\title{On the Consistency of the Crossmatch Test}

\author{
Ery Arias-Castro\footnote{Department of Mathematics, University of California, San Diego, CA, USA}
\and
Bruno Pelletier\footnote{D\'epartement de Math\'ematiques, IRMAR -- UMR CNRS 6625, Universit\'e Rennes II, France}
}
\date{}
\maketitle

\begin{abstract}
\cite{rosenbaum2005exact} proposed the crossmatch test for two-sample goodness-of-fit testing in arbitrary dimensions.  
We prove that the test is consistent against all fixed alternatives.
In the process, we develop a general consistency result based on \citep{henze1999multivariate} that applies more generally. 
\end{abstract}

\section{Introduction} \label{sec:intro}
Two-sample goodness-of-fit testing is an important line of research in statistics.  
We consider the continuous setting where we observe two independent samples, $X_1, \dots, X_m$ and $Y_1, \dots, Y_n$ in $\bbR^d$.  
Classical approaches include the Kolmogorov-Smirnov test \citep{kolmogorov1933sulla,smirnov1939estimation}, the number-of-runs test \citep{MR0002083}, and the longest-run test \citep{MR0004453}.  These procedures were originally designed for real-valued observations ($d=1$).  Over the years, a number of approaches that apply to vector-valued observations ($d > 1$) have been suggested.  
Among these are a class of tests based on graph constructions.  This goes back at least to the work of \cite{friedman1973nonparametric}.  
Their method is based on counting the number of $X$'s among the $K$-nearest neighbors each observation in the combined sample.  
See also \citep{rogers1976some} and more recently \citep{hall2002permutation}.
Although it does not cover all the possibilities, many of the subsequent proposals can be framed as follows.  Let $t = m+n$ denote the total sample size, and let $\cG$ be a directed graph with node set the combined sample $\{Z_1, \dots, Z_t\}$, with $Z_k = X_k$ if $k \le m$ and $Z_k = Y_{k - m}$ if $k > m$.  We write $Z_i \to Z_j$ when there is an edge from $Z_i$ to $Z_j$ in $\cG$.  Consider rejecting for {\em small} values of
\beq\label{crossmatch}
\chi_\cG(\bZ) = \# \{i \le m, j > m : Z_i \to Z_j\} + \# \{i \le m, j > m : Z_j \to Z_i\},
\eeq
which is the number of neighbors in the graph from different samples.
If the graph $\cG$ is the $K$-nearest neighbor graph --- where $Z_i \to Z_j$ if $Z_j$ is among the $K$-nearest neighbors of $Z_i$ in Euclidean distance --- and we assume that all the $Z$'s are distinct, then the resulting test is that of \cite{schilling1986multivariate}, a special case of the general approach of \cite{friedman1973nonparametric}. 
If the graph $\cG$ is a minimum spanning tree (starting with the complete graph weighted by the Euclidean distances), then the resulting test is the multivariate runs test of \cite{friedman1979multivariate}.  

If the graph $\cG$ is a minimum distance matching, then the resulting test is that of \cite{rosenbaum2005exact}, which is the method that we analyze in the present paper --- and possibly the latest in this line.
\cite{rosenbaum2005exact} calls his method the crossmatch test.  
The graph $\cG$ is a minimal distance matching of the combined sample, where the distances are understood here as being the Euclidean distances.  
In detail, the following optimization problem is solved
\beq\label{matching}
\min_{\sigma} \sum_{k=1}^t \|Z_k - Z_{\sigma(k)}\|,
\eeq
where the minimization is over permutations $\sigma$ of $[t] := \{1,\dots, t\}$ of order 2 with at most one fixed point, meaning, $\sigma(\sigma(k)) = k$ for all $k$ and $\sigma(k) \ne k$ for all $k$ except at most one.  (Note that there is a fixed point if and only if $t$ is odd.)  Choosing a solution $\hat\sigma$ at random if there are several, the graph $\cG$ is then defined as the undirected graph with vertex set $[t]$ and edge set $\{(k,\hat\sigma(k)) : \sigma(k) \ne k\}$.  

Among methods based on graph constructions, the crossmatch remains poorly understood. 
\cite{rosenbaum2005exact} derives the null distribution --- which is also studied in \cite{heller2010sensitivity}.  While most other tests need to be calibrated by permutation, the crossmatch test has the nice feature of having a null distribution (which coincides with its permutation distribution) available in closed form. 
However, the power properties of the test, and in particular its consistency, have not been established.

In the present paper, we prove that the crossmatch is consistent against all fixed alternatives.
The proof is based on the arguments developed by \cite{henze1999multivariate} to establish the consistency of the multivariate runs test of \cite{friedman1979multivariate}.
Based on their work, we develop a general result that applies to graph-based tests where the graph has short edges and bounded degree, which in particular applies to other matchings and to the nearest-neighbor graph method of \cite{schilling1986multivariate}.

\section{Setting}

We observe two independent samples, $X_1, \dots, X_m$ IID with density $f$ and $Y_1, \dots, Y_n$ IID with density $g$, both with respect to the Lebesgue measure on $\bbR^d$.  
The goal is to test
\beq
H_0 : f = g \text{ versus } H_1 : f \ne g.
\eeq
(Of course, this is understood modulo a set of measure zero.) 
We assume that the sample sizes are comparable in the sense that
\beq\label{usual}
\frac{m}{m+n} \to p \in (0,1).
\eeq

Formally, a graph construction is a function $\cG$ defined on the finite (unordered) subsets of $\bbR^d$, where for $z_1, \dots, z_t \in \bbR^d$, $\cG(z_1,\dots, z_t)$ is a simple\footnote{A simple graph has no multi-edges and no self-loops.  Such a graph (if unweighted) can be represented by its adjacency matrix, having only 0's and 1's, with all 0's on the diagonal.}
directed graph with $k$ nodes.  Recalling the definition of the $Z_k$'s in the Introduction, the method starts by constructing the graph based on the combined sample, which means computing $\cG(\bZ)$.  Once this is done, the statistic $\chi_\cG(\bZ)$, defined in \eqref{crossmatch}, is computed.  The test rejects for {\em small} values of $\chi_\cG(\bZ)$ and calibration is typically done by permutation.  

\begin{rem}
Although the literature on graph-based goodness-of-fit testing in arbitrary dimensions is silent on the topic, the setting exhibits a typical curse of dimensionality, as we discuss in a forthcoming paper  
\citep{gof-graph}.  We therefore assume that $d$ is constant.
\end{rem}

\section{Almost sure convergence for general graphs}
For $\bz = \{z_1, \dots, z_t\}$ in $\bbR^d$, let $\Delta^{\rm out}(z_k; \cG(\bz))$ (resp.~$\Delta^{\rm in}(z_k; \cG(\bz))$) denote the out-degree (resp.~in-degree) of $z_k$ in graph $\cG(\bz)$.  We assume that $\cG$ has degree bounded by $\delta_0$, meaning that, for any set of points $\bz = \{z_1, \dots, z_t\}$,
\beq\label{max-deg}
\Delta^{\rm out}(z_t; \cG(\bz)) \vee \Delta^{\rm in}(z_t; \cG(\bz)) \le \delta_0.
\eeq
We also assume that the out-degree is essentially constant and that long edges are essentially absent.  Specifically, following  \citep{henze1999multivariate}, let $\phi, \phi_1, \phi_2, \dots$ be any density functions with same support such that $\phi_t/\phi \to 1$ uniformly on $\{\phi > 0\}$.  
Let $\bZ_t = \{Z_{1,t}, \dots, Z_{t,t}\}$ be IID with density $\phi_t$.  Then, under these circumstances, we assume that the out-degree satisfies
\beq\label{limit-deg}
\Delta^{\rm out}(Z_{1,t}; \cG(\bZ_t)) \to^P \delta, \quad t \to \infty,
\eeq
where $\delta > 0$ only depends on $\cG$.
We also assume that edges of length much larger than $t^{-1/d}$ are unlikely, in the sense that
\beq\label{local}
\lim_{a \to \infty} \lim_{t \to \infty} \P\big[\exists k \in [t] : Z_{1,t} \to Z_{k,t} \text{ in } \cG(\bZ_t) \text{ and } \|Z_{1,t} - Z_{k,t}\| > a t^{-1/d}\big] = 0.
\eeq 
We note that $t^{-1/d}$ is the order of magnitude of the distance between a sample point and its closest neighbor in the sample.

\begin{rem}
The conditions above can be shown to cover not only the minimum spanning tree (as shown by \cite{henze1999multivariate}), but also nearest-neighbor graphs \citep{schilling1986multivariate} and general matchings (our main interest here) as shown in \secref{matchings}.
\end{rem}

\begin{thm} \label{thm:main}
Assume the graph construction $\cG$ satisfies these assumptions.
Then, in the limiting regime \eqref{usual}, 
\beq
\frac{\chi_\cG(\bZ)}{m+n} \to 2\delta \int \frac{p(1-p) f(z) g(z)}{pf(z) + (1-p)g(z)} {\rm d}z, \quad \text{almost surely.}
\eeq
\end{thm}

The proof of \thmref{main} is exactly the same as that of Theorem~2 in \citep{henze1999multivariate}, treating out-edges and in-edges separately, and with Proposition~1 there replaced with the following.   
As in the proof of Theorem~2 in \citep{henze1999multivariate}, Proposition~\ref{prp:main} is used in the proof of Theorem~\ref{thm:main} with the choice of $\phi=pf+(1-p)g$ and $\phi_t = (mf+ng)/t$, with $m$ and $n$ implicitly parameterized by $t$.  (Recall that $t = m+n$ is the total sample size.)  

\begin{prp} \label{prp:main}
Let $\phi, \phi_1, \phi_2, \dots$ be any density functions with same support such that $\phi_t/\phi \to 1$ uniformly on $\{\phi > 0\}$, let $\bZ_t = \{Z_{1,t}, \dots, Z_{t,t}\}$ be IID with density $\phi_t$, and assume the conditions of \thmref{main} hold.
Let $h : \bbR^d \times \bbR^d \mapsto [0,1]$ be measurable and such that almost any $z \in \bbR^d$ is a Lebesgue continuity point of $h(z,\cdot) \phi(\cdot)$.  
Then
\beq
\lim_{t \to \infty} \frac1t \E \mathop{\sum\sum}_{1 \le i < j \le t} h(Z_{i,t}, Z_{j,t}) \IND{Z_{i,t} \to Z_{j,t} \text{ in } \cG(\bZ_t)} = \tfrac12 {\delta} \int h(z,z) \phi(z) {\rm d}z.
\eeq
\end{prp}

Proposition~1 in \citep{henze1999multivariate} shows this when $\cG$ is the minimum spanning tree based on properties obtained in \citep{aldous1992asymptotics}.
The proof of \prpref{main} still borrows much from that of Proposition~1 in \citep{henze1999multivariate}, although it is simpler here because we work under more specific assumptions which \citep{henze1999multivariate} verify for the minimum spanning tree along the way.

\begin{proof}
For any $a > 0$, we have
\beq\begin{split}
\frac1t \E \mathop{\sum\sum}_{1 \le i < j \le t} h(Z_{i,t}, Z_{j,t}) \IND{Z_{i,t} \to Z_{j,t}  \text{ in } \cG(\bZ_t)} 
&= \frac12 \E \sum_{k=2}^t h(Z_{1,t}, Z_{k,t}) \IND{Z_{1,t} \to Z_{k,t} \text{ in } \cG(\bZ_t)} \\
&= \frac12\big[ \mathfrak{A} + \mathfrak{B} + \mathfrak{C} \big],
\end{split}\eeq
where
\beq\begin{split}
\mathfrak{A} &:= \E \sum_{k=2}^t  h(Z_{1,t}, Z_{1,t}) \IND{Z_{1,t} \to Z_{k,t} \text{ in } \cG(\bZ_t)}, \\
\mathfrak{B} &:= \E \sum_{k=2}^t (h(Z_{1,t}, Z_{k,t}) - h(Z_{1,t}, Z_{1,t}))\IND{Z_{1,t} \to Z_{k,t} \text{ in } \cG(\bZ_t) \text{ and } \|Z_{1,t} - Z_{k,t}\| > a t^{-1/d}}, \\
\mathfrak{C} &:= \E \sum_{k=2}^t (h(Z_{1,t}, Z_{k,t}) - h(Z_{1,t}, Z_{1,t})) \IND{Z_{1,t} \to Z_{k,t} \text{ in } \cG(\bZ_t) \text{ and } \|Z_{1,t} - Z_{k,t}\| \le a t^{-1/d}}.
\end{split}\eeq

For the first term,
\beq\begin{split}
\mathfrak{A}
&= \E h(Z_{1,t}, Z_{1,t}) \Delta^{\rm out}(Z_{1,t}; \cG(\bZ_t)) \\
&\sim \delta \E h(Z_{1,t}, Z_{1,t}) 
= \delta \int h(z,z) \frac{\phi_t(z)}{\phi(z)} \phi(z) {\rm d}z 
\to \delta \int h(z,z) \phi(z) {\rm d}z, \quad t \to \infty.
\end{split}\eeq
In the third line we first used \eqref{limit-deg} and dominated convergence (enabled by the fact that $0 \le h \le 1$); and then dominated convergence again (enabled by the fact that $0 \le h \le 1$ and $0 \le \phi_t/\phi \le 2$ eventually).

For the second term 
\beq\begin{split}
|\mathfrak{B}|
&\le \E \sum_{k=2}^t \IND{Z_{1,t} \to Z_{k,t} \text{ in } \cG(\bZ_t) \text{ and } \|Z_{1,t} - Z_{k,t}\| > a t^{-1/d}} \\
&= \P\big[\exists k \in [t] : Z_{1,t} \to Z_{k,t} \text{ in } \cG(\bZ_t) \text{ and } \|Z_{1,t} - Z_{k,t}\| > a t^{-1/d}\big],
\end{split}\eeq
using the fact that $0 \le h \le 1$.  Hence, $\lim_{a \to \infty} \lim_{t \to \infty} \mathfrak{B} = 0$ by \eqref{local}.

For the third term, we have $|\mathfrak{C}| \le \int \psi_t(z) \phi_t(z) {\rm d}z$, where
\beq\begin{split}
\psi_t(z) 
&:= \E \sum_{k=2}^t |h(z, Z_{k,t}) - h(z,z)| \IND{z \to Z_{k,t} \text{ in } \cG(z, Z_{2,t}, \dots, Z_{t,t}) \text{ and } \|z - Z_{k,t}\| \le a t^{-1/d}}.
\end{split}\eeq
Note that, by \eqref{max-deg}, 
\beq
\psi_t(z) \le \E \Delta^{\rm out}(z; \cG(z, Z_{2,t}, \dots, Z_{t,t})) \le \delta_0.
\eeq
We now show that $\psi_t(z) \to 0$ as $t \to \infty$ for almost all $z$'s and for any fixed $a > 0$, which will imply that $\lim_{t \to \infty} \mathfrak{C} = 0$ by dominated convergence and our assumption on $\phi_t$.  
Indeed, we have
\beq\begin{split}
\psi_t(z) 
&\le (t-1) \E |h(z, Z_{2,t}) - h(z,z)| \IND{\|z - Z_{2,t}\| \le a t^{-1/d}} \\
&= (t-1) \int_{B(z, a t^{-1/d})} \big|h(z,u)\phi_t(u) - h(z,z) \phi_t(z) + h(z,z)\phi_t(z) - h(z,z)\phi_t(u) \big| {\rm d}u \\
&\le t \int_{B(z, a t^{-1/d})} \big|h(z,u)\phi(u) - h(z,z) \phi(z)\big| {\rm d}u + t \int_{B(z, a t^{-1/d})} |\phi(z) - \phi(u)| {\rm d}u \\
&\quad + 2 t \int_{B(z, a t^{-1/d})} |\phi_t(u) - \phi(u)| {\rm d}u +  2\omega_d a^d \big|\phi_t(z) - \phi(z)\big|,
\end{split}\eeq
where $\omega_d$ denotes the volume of the unit ball in $\bbR^d$, and using the triangle inequality and the fact that $h$ has values in $[0,1]$.
Noting that the Lebesgue measure (the usual volume) of $B(z, a t^{-1/d})$ is equal to $\omega_d a^dt^{-1}$, we see that the first integral in the last line converges to zero as $t \to \infty$ when $z$ is a Lebesgue continuity point of $h(z, \cdot) \phi(\cdot)$, and the same is true of the second integral if $z$ is a Lebesgue continuity point of $\phi(\cdot)$.  
We note that almost all points $z$ satisfy these properties.  This follows from our assumptions on $h$ and the fact that almost all points are Lebesgue continuity points of a given integral function (including $\phi$).  
Moreover, because of the uniform convergence of $\phi_t$ towards $\phi$, the third integral converges to zero, and the last term converges to zero for the same reason.

Thus, by taking limits as $t \to \infty$ first and then as $a \to \infty$, we conclude.
\end{proof}

\section{Almost sure convergence for matchings}
\label{sec:matchings}

We now prove that certain kinds of matchings --- including the matching \eqref{matching} --- satisfy the conditions of \thmref{main}.
Let $\bz=\{z_1,\dots,z_t\}$ be any set of points in $\bbR^d$.
Since a matching results in all vertices having exactly one neighbor, except for exactly one of them if $t$ is odd, the resulting graph $\cG(\bz)$ has degree bounded by 1, so that \eqref{max-deg} holds with $\delta_0 = 1$. 
Assume now $\bZ = \{Z_1, \dots, Z_t\}$ are IID from a diffuse distribution on $\bbR^d$. 
Then $\Delta^{\rm out}(Z_1; \cG(\bZ))$ follows a Bernoulli distribution with parameter $1 - \tfrac1t \IND{t \text{ odd}}$, so that \eqref{limit-deg} is satisfied with $\delta = 1$.
It remains to establish \eqref{local}, meaning that long edges are rare, which is intuitively natural since matchings aim at minimizing pairing distances.

\subsection{Optimal matchings}
\label{sec:optimal}
We start with a matching which applies to $\bZ = \{Z_1, \dots, Z_t\} \subset \bbR^d$ and is of the form
\beq\label{general-matching}
\min_{\sigma} \Lambda(\bZ; \sigma), \quad \Lambda(\bZ; \sigma) := \sum_{k=1}^t \lambda(\|Z_k - Z_{\sigma(k)}\|),
\eeq
where $\lambda : [0, \infty) \to \bbR$ is non-decreasing and the minimization is as in \eqref{matching}.  Examples include $\lambda(b) = b^\alpha$ where $\alpha > 0$, but one could imagine taking $\lambda$ such that $\lim_{b \to \infty} \lambda(b) < \infty$ for robustness.  Below, we assume that 
\beq\label{lambda}
\lambda(b) \asymp b^\alpha, \quad b \to 0, \quad \text{for some } \alpha \in (0, d).
\eeq
This condition includes the case $\lambda(b) = b^\alpha$ when $\alpha \in (0, d)$, and in particular the original matching \eqref{matching} in dimension $d \ge 2$.

\begin{prp} \label{prp:matching}
Let $\cG$ be the result of any matching of the form \eqref{general-matching} with $\lambda$ satisfying \eqref{lambda}.  
Then $\cG$ fulfills \eqref{local}.
\end{prp}

\begin{proof}
Let $\phi, \phi_1, \phi_2, \dots$ be any density functions with same support such that $\phi_t/\phi \to 1$ uniformly on $\{\phi > 0\}$.  
Let $\bZ_t = \{Z_{1,t}, \dots, Z_{t,t}\}$ be IID with density $\phi_t$.
Take $\eps > 0$.  Take $r > 0$ such that $\int_A \phi > 1 - \eps$, where $A := [-r,r]^d$.  By our assumptions, there is $t_0$ such that $\int_A \phi_t > 1 - 2\eps$ for all $t \ge t_0$.
Let $\hat\sigma \in \argmin_\sigma \Lambda(\bZ_t; \sigma)$ and define a matching $\tilde\sigma$ as follows.  
Let $K_t = \{k \in [t] : Z_{k,t} \notin A\}$. 
For $k \in K_t \cup \hat\sigma(K_t)$, let $\tilde\sigma(k) = \hat\sigma(k)$.
The indices $k \notin K_t \cup \hat\sigma(K_t)$ are matched between themselves as follows.  Let $s_0$ be the largest integer such that $2^{s_0 d} \le t$.
Consider a regular partition of $A$ into $2^{s_0 d}$ hypercubes, each of side length $2r2^{-s_0}$, and therefore diameter $2^{-s_0} 2r\sqrt{d}$.
Match points within each bin arbitrarily, for example, by solving \eqref{general-matching} within that bin.  Note that each pair contributes at most $\lambda(2^{-s_0} 2r\sqrt{d})$ to $\Lambda(\bZ_t; \tilde\sigma)$.  Remove the matched points.  Since this leaves at most one point per bin, there are at most $2^{s_0 d}$ points left.  Form a regular partition of $A$ into $2^{(s_0-1) d}$ hypercubes and match points within each bin.  Each pair now contributes at most $\lambda(2^{-(s_0-1)} 2r\sqrt{d})$ and, after removing the matched points, there are at most $2^{(s_0-1) d}$ points left.  Continuing in this fashion defines $\tilde\sigma$ and we have
\beq\begin{split}
\Lambda(\bZ_t; \tilde\sigma) 
&\le \sum_{k \in K_t \cup \hat\sigma(K_t)} \lambda(\|Z_{k,t} - Z_{{\hat\sigma(k)},t}\|) \\
&\quad + t \lambda(2^{-s_0} 2r\sqrt{d}) + 2^{s_0 d} \lambda(2^{-(s_0-1)} 2r\sqrt{d}) + 2^{(s_0-1) d} \lambda(2^{-(s_0-2)} 2r\sqrt{d}) + \cdots + 2^d \lambda(2r\sqrt{d}). 
\end{split}\eeq
By the fact that $\hat\sigma$ minimizes \eqref{general-matching}, 
\beq
\Lambda(\bZ_t; \tilde\sigma) \ge \Lambda(\bZ_t; \hat\sigma) = \sum_{k \in K_t \cup \hat\sigma(K_t)} \lambda(\|Z_{k,t} - Z_{{\hat\sigma(k)},t}\|) +  \sum_{k \notin K_t \cup \hat\sigma(K_t)} \lambda(\|Z_{k,t} - Z_{{\hat\sigma(k)},t}\|).  
\eeq
By our condition \eqref{lambda} on $\lambda$, there is $C > c > 0$ such that
\begin{equation}
\label{eq:lambdafunc}
c b^\alpha \le \lambda(b) \le C b^\alpha \quad\text{ for all $b \in [0,1]$.}
\end{equation}
Let $j_r = \min\{j : 2^{-j} 2r\sqrt{d} \le 1\}$.  
Below, $B_1, B_2, \dots$ are positive functions of $\eps, d, \alpha, c, C$, and recall that $r$ is a function of $\eps$.   
With  $2^{s_0 d} \le t \le 2^{(s_0+1)d}$, we get
\beq\begin{split}
\sum_{k \notin K_t \cup \hat\sigma(K_t)} \lambda(\|Z_{k,t} - Z_{{\hat\sigma(k)},t}\|) 
&\le \sum_{j=j_r}^{s_0} 2^{(j+1) d} C(2^{-j} 2r\sqrt{d})^\alpha + \sum_{j=0}^{j_r-1} 2^{(j+1) d} \lambda(2^{-j} 2r\sqrt{d}) \label{eq:sum1}\\
&\le B_2 2^{(d-\alpha) s_0} + B_1 
\le B_2 t^{1 - \alpha/d} + B_1
\le 2 B_2 t^{1 - \alpha/d},
\end{split}\eeq
when $t \ge B_3 := (B_1/B_2)^{1/(1-\alpha/d)}$.
For $a > 0$, let 
\beq
L_{a,t} = \big\{k \notin K_t \cup \hat\sigma(K_t) : \lambda(\|Z_{k,t} - Z_{{\hat\sigma(k)},t}\|) \ge \lambda(a t^{-1/d})\big\}.
\eeq  
Using \eqref{eq:lambdafunc}, we have
\beq
\sum_{k \notin K_t \cup \hat\sigma(K_t)} \lambda(\|Z_{k,t} - Z_{{\hat\sigma(k)},t}\|) 
\ge |L_{a,t}| \lambda(a t^{-1/d})
\ge |L_{a,t}| c (a t^{-1/d})^\alpha, \label{eq:sum2}
\eeq
when $t \ge a^d$, in which case, by combing \eqref{eq:sum1} and \eqref{eq:sum2}, it follows that $|L_{a,t}|/t \le \frac{2}{c} B_2 a^{-\alpha} =: B_4a^{-\alpha}$.

We have
\beq
\P_t(\|Z_{1,t} - Z_{{\hat\sigma(1)},t}\| \ge a t^{-1/d}) \le \P_t(1 \in K_t \cup \hat\sigma(K_t)) + \P_t(1 \in L_{a,t}),
\eeq
where we have used the fact that $\lambda$ is non-decreasing by assumption.
On the one hand, 
\beq
\P_t(1 \in L_{a,t}) 
= \frac1t \sum_{i =1}^t \P_t(i \in L_{a,t}) 
= \frac1t \sum_{i =1}^t \E_t[\IND{i \in L_{a,t}}] 
= \frac1t \E_t[|L_{a,t}| ]
\le B_4 a^{-\alpha}.
\eeq
On the other hand, by construction of $A$, and then $K_t$,
\beq
\P_t(1 \in K_t \cup \hat\sigma(K_t)) \le \P_t(1 \in K_t) + \P_t(\hat\sigma(1) \in K_t) \le 2(2\eps) = 4\eps.
\eeq
Thus, for $a \ge (B_4/\eps)^{1/\alpha}$ and $t \ge B_3 \vee a^d$,
\beq
\P_t(\|Z_{1,t} - Z_{{\hat\sigma(1)},t}\| \ge a t^{-1/d}) \le 5 \eps.
\eeq
Since $\eps$ is arbitrary, we can conclude that \eqref{local} holds.
\end{proof}

We may thus apply \thmref{main} to obtain the following.

\begin{cor} \label{cor:main}
Let $\cG$ be the result of any optimal matching of the form \eqref{general-matching} with $\lambda$ satisfying \eqref{lambda}.  
Then, in the limiting regime \eqref{usual}, 
\beq\label{main}
\frac{\chi_\cG(\bZ)}{m+n} \to \int \frac{2 p(1-p) f(z) g(z)}{pf(z) + (1-p)g(z)} {\rm d}z, \quad \text{almost surely.}
\eeq
\end{cor}

\subsection{Greedy matching} \label{sec:greedy}

By greedy matching we mean the following procedure: match the closest pair of points (in Euclidean distance), remove them from the sample, and repeat.  (Ties are broken arbitrarily.)  
Let $\sigma_\circ$ denote this greedy matching.
This can be seen as a greedy minimization strategy for \eqref{general-matching}.  While a typical implementation of \eqref{general-matching} has complexity $O(t^3)$, this greedy procedure has complexity $O(t^{3/2} \log t)$.  See the discussion in \citep{avis1988probabilistic}.  There, it is proved that, if $Z_1, \dots, Z_t$ are IID from a distribution with compact support on $\bbR^d$ (with $d \ge 2$),
\beq
\sum_{k \in [t]} \|Z_k - Z_{\sigma_\circ(k)}\| = O(t^{1 - 1/d}).
\eeq
(In fact, their result is much sharper.)  Although this result is obtained when the underlying distribution does not change with $t$, it can be seen to easily extend to the setting of \thmref{main} as long as $\phi$ has compact support.

With this in place, we can reason as we did in \secref{optimal} to find that the graph $\cG$ resulting from the greedy matching $\sigma_\circ$ also satisfies \eqref{local}.
We may thus apply \thmref{main} to derive the following.

\begin{cor} \label{cor:main-greedy}
Let $\cG$ be the result of the greedy matching.  
Then, assuming that the underlying distributions $f$ and $g$ have compact support, and in the limiting regime \eqref{usual}, the limit \eqref{main} holds.
\end{cor}

\section{Consistency against all (fixed) alternatives}

We start with describing the behavior of the test statistic $\chi_\cG(\bZ)$ under the null hypothesis where $f = g$.  \cite{rosenbaum2005exact} derives the null distribution of $\chi_\cG$ in closed form (which happens to be equal to its permutation distribution) and finds the exact moments to be\footnote{Note that in our definition there is an extra factor of 2.}
\beq\label{moments}
\E[\chi_\cG(\bZ)] = \frac{2 mn}{t-1} \sim 2p(1-p) t, \quad
\Var[\chi_\cG(\bZ)] = \frac{8 m(m-1)n(n-1)}{(t-1)^2 (t-3)} \sim 8 p^2(1-p)^2 t,
\eeq
where the limits are as $t \to \infty$ in the regime \eqref{usual}.
Clearly, this is true for any matching.
To prove consistency, this is all we need together with \corref{main} (or \corref{main-greedy}), although \cite{rosenbaum2005exact} also shows that the null distribution is asymptotically normal.  

\begin{cor} \label{cor:consistent}
Let $\cG$ be the result of any matching of the form \eqref{general-matching} with $\lambda$ satisfying \eqref{lambda}.  
For any sequence $\eta_t \to 0$ such that $\eta_t \gg 1/\sqrt{t}$, the test with rejection region $\{\frac1t \chi_\cG(\bZ) < 2 p(1-p) - \eta_t\}$ is consistent against all (fixed) alternatives.
This remains true if $\cG$ is the result of the greedy matching and the underlying distributions have compact support.
\end{cor}

\begin{proof}
By Chebyshev's inequality and the expression for the null moments \eqref{moments}, we see that the test has asymptotic size 0.  
Under a fixed alternative, $f \ne g$, \corref{main} gives 
\beq
\frac1t \chi_\cG(\bZ) \to 2 p(1-p) \int \frac{f(z) g(z)}{pf(z) + (1-p)g(z)} {\rm d}z, \quad \text{almost surely}.
\eeq
As \cite{henze1999multivariate} argue, based on work of \cite{gyorfi}, this limit is strictly smaller when $f \ne g$ than that when $f = g$ (that latter being equal to $2 p (1-p)$).  Hence, the test has asymptotic power 1.
\end{proof}

\section*{Acknowledgments}
We would like to thank Venkatesh Saligrama for introducing us to the work of \cite{batu2013testing}, and Kasturi Varadarajan and Jeff Phillips for helpful discussions and references regarding the computation of non-bipartite matchings.
This work was partially supported by the US Office of Naval Research (N00014-13-1-0257)

\bibliographystyle{chicago}
\bibliography{crossmatch}

\end{document}